\documentclass[12pt,a4paper]{amsart}

\usepackage[ruled,vlined,boxed,norelsize]{algorithm2e}

\usepackage{fullpage}
\usepackage{amsmath, amssymb}
\usepackage{amsfonts}
\usepackage{amscd}
\usepackage{enumerate}
\usepackage{amsthm}
\usepackage[T1]{fontenc}
\usepackage[latin1]{inputenc} 
\usepackage[T1]{fontenc}
\usepackage{url}
\usepackage{hyperref}

\def\color[#1]#2{}

\makeatletter \gdef\th@break{\normalfont\slshape
  \def\@begintheorem##1##2{\item[%
       \rlap{\vbox{\hbox{\hskip \labelsep\theorem@headerfont ##1\ ##2}%
                   \hbox{\strut}}}]}%
\def\@opargbegintheorem##1##2##3{%
  \item[\rlap{\vbox{\hbox{\hskip \labelsep \theorem@headerfont
                     ##1\ ##2\ ##3}%
                    \hbox{\strut}}}]}}
\makeatother

\newtheorem{theorem}{Theorem}[section]

\newtheorem{definition}{Definition}[section]
\newtheorem{proposition}{Proposition}[section]
\newtheorem{corollary}{Corollary}[section]
\newtheorem{remark}{Remark}

\newtheorem{lemma}{Lemma}[section]

\newcommand{\FF}{\mathbb{F}}
\newcommand{\CC}{\mathbb{C}}
\newcommand{\HH}{\mathbb{H}}

\newcommand{\PP}{\mathbb{P}}

\newcommand{\ZZ}{\mathbb{Z}}

\newcommand{\bfA}{\mathbf{A}}

\newcommand{\bfM}{\mathbf{M}}

\newcommand{\calC}{\mathcal{C}}

\newcommand{\calS}{\mathcal{S}}

\def\Z{\mathbb{Z}}
\def\C{\mathbb{C}}
\def\P{\mathbb{P}}

\def\Sp{\textrm{Sp}}

\def\Jac{\textrm{Jac}}

\def\Pic{\textrm{Pic}}

\def\eps{\epsilon}

\renewcommand{\Im}{\mathop{\mathrm{Im}}\nolimits}
\newcommand{\Ch}[2]{\begin{bmatrix} #1 \\ #2 \end{bmatrix}}
\def\ex{\boldsymbol{e}}

\newcommand{\car}[2]{%
\left[%
\begin{smallmatrix}%
\displaystyle#1\\%
\displaystyle#2\end{smallmatrix}\right]}

\title{A new proof of a Thomae-like formula for non hyperelliptic
  genus 3 curves}
\author{Enric Nart,  Christophe Ritzenthaler}

\thanks{The first author acknowledges support
from grant MTM2103-40680-P from the Spanish MEC. The second author acknowledges support by grant ANR-09-BLAN-0020-01, and by
  the research programme \emph{Investissements d'avenir} (ANR-11-LABX-0020-01)
  of the Centre Henri Lebesgue. }

\begin{document}
\maketitle

\begin{abstract}
We discuss Weber's formula which gives the quotient of
two Thetanullwerte for a  plane smooth quartic in terms of the
bitangents. In particular, we show how it can easily be
derived from the Riemann-Jacobi formula.
\end{abstract}

\section{introduction}
Let $g>0$ be an integer and $\bfM_g$ (resp. $\bfA_g$) be the coarse moduli
space of smooth, irreducible and projective curves of genus $g$
(resp. principally polarized abelian varieties of dimension $g$) over
$\CC$. These two important moduli spaces are related through the
\emph{Torelli map} $j$ which associates to the isomorphism class of
a curve, the isomorphism class of its Jacobian with its canonical polarization.
\emph{Thomae-like formulae} can be seen as an explicit description of
the Torelli map. Indeed, as Mumford showed in \cite{mumford-def}, a principally
polarized abelian variety can be written down as intersection of
explicit
quadrics in a projective space.  Now, the coefficients of these quadrics are determined by a
certain projective vector of constants called \emph{Thetanullwerte} (or
Thetanulls) that we shall denote $\vartheta [q](\tau)$ (see Section
\ref{RJformula}) where $\tau$ is a
Riemann matrix for a specific choice of bases of regular differentials
and homology and $[q]$ is a \emph{characteristic}. Thomae-like
formulae express these constants (or quotients raised to a certain power) in terms of the geometry of the curve. In the case of
a hyperelliptic curve $y^2=\prod_{i=1}^{2g+2} (x-\alpha_i)$, Thomae
himself \cite[p.218]{thomae2}  found that
$$\vartheta[q](\tau)^4 = (2 i \pi)^{-2g} \cdot \det(\Omega_1)^2 \cdot \prod_{i,j \in U} (\alpha_i-\alpha_j)$$
where $\Omega_1$ is the first half of a period matrix and $U$ is a set of indices depending on the
characteristic $[q]$. This formula, which we call the \emph{absolute
  Thomae formula} has then been reproved by \cite{fuchs, bolza2,
   fay} using the  variational method. A simpler
version, which we call the \emph{relative Thomae formula}, expressing
the quotient $\vartheta[q](\tau)^8/\vartheta[q'](\tau)^8$ was then
achieved in \cite{zariski, mumford-tata2,eisenmann} using elementary arguments. Note that this
formula, which involves only the roots $\alpha_i$ is
generally sufficient to recover the Jacobian and can moreover be
worked out over an arbitrary field \cite{shepherdbarron-2008}.  The issue of
  finding the correct $8$th roots of the quotients is considered for $g=1,2$ in \cite{cosset}
  and can be simply solved over $\CC$  by computing the Thetanullwerte with a weak precision.\\

 In the last 20 years, the subject came to a renaissance thanks to its applications, on one
side, to theoretical physics (\cite{smirnov,radul1} and the references
of \cite{enolski}) and on the other side to
cryptography (\cite{wengh,agmri,LL06,lubicz-robert}). With the first
applications in mind, the authors of \cite{radul1,radul2,nakayashiki,enolski,gonzalezdiez,hattori} have
been able to find absolute or relative versions of Thomae-like formula
in the case of $Z_N$-curves, i.e. curves of the form $y^N = f(x)$. As
for cryptography, the use of Thomae-like formula for the AGM point counting
algorithm  in the spirit of Mestre lead the second author to dig out a
relative formula for non hyperelliptic genus $3$ curves due to Weber
\cite{weber}. This is the formula we will consider in this article
(see Theorem \ref{eq:jacobi}). Note that Thomae-like formulae are also naturally connected to the Schottky problem of
characterizing the locus of $j(\bfM_g)$ in $\bfA_g$ and explicit
solutions to reconstruct a curve from its Jacobian can be found in
\cite{rosenhain} for the genus $2$ case, in \cite{takase,koizumi2}
for the general hyperelliptic case and in \cite{weber, guardia} for
the non hyperelliptic genus $3$ case.\\

The combinatoric behind Weber's formula for non hyperelliptic genus
$3$ curves is more involved than in the hyperelliptic case as
there is no obvious ordering of the geometric data (the $28$
bitangents) unlike the roots $\alpha_i$ on the
projective line. The Ancients solved this issue by the use of
\emph{Aronhold bases}. We recall this theory and derive some useful
lemmas in Section \ref{sec:general}. In order to  formulate a
coordinate-free result,  we  consider these notions in the
framework of
quadratic forms over $\FF_2$ as in \cite{grossharris}. In Section
\ref{sec:weber}, we give an overview and a simplification of 
Weber's original proof  in order to compare it with ours.  We 
want to point out (see Remark \ref{rem:algo}) that Weber's proof may
lead to an algorithm for computing Thetanullwerte in arbitrary genus in the spirit
of \cite{shepherdbarron-2008}. \\
In Section \ref{sec:our}, we present our proof. It is shorter and 
 based on  \emph{Riemann-Jacobi formula} (see Corollary
\ref{cor:jacobi}). This formula  gives a
link between \emph{Jacobian Nullwerte} (see Definition
\ref{def:jacobian}) and certain products of Thetanullwerte. Now,
 Jacobian Nullwerte are  determinants of bitangents
(Corollary \ref{link}) up to  multiplicative constants. We use  an
elementary combinatoric operation to isolate one Thetanullwert, get
rid of the multiplicative constants in the quotient and
then get Weber's formula up to a sign (which is left unspecified in the
Riemann-Jacobi formula). In Section \ref{sec:sign}, we find the sign using
a low precision computation and a transformation formula due to
Igusa.\\

\noindent
{\bf Acknowledgment.} The authors are grateful to Riccardo
Salvati Manni for its comments and support.

\section{Review on Aronhold sets, fundamental systems and
  Riemann-Jacobi formula}
We start with some general definitions and results on combinatorics of
theta characteristics and Aronhold systems in the spirit of
\cite{grossharris}. We then review some basic notions on theta
functions (see for instance \cite[Chap.I]{farkas}) and the
Riemann-Jacobi formula. We end up with some general results about the
links between a curve and its Jacobian.

\subsection{Quadratic forms over $\FF_2$} \label{sec:general} \label{rem:iden}
Let $g \geq 1$ be an integer and $V$ be a vector space of dimension
$2g$ over $\FF_2$. We fix a nondegenerate alternating form $\langle ,
\rangle$ on $V$ and we say that $q : V \to \FF_2$ is a \emph{quadratic form}
on $V$ if for all $u,v \in V$
$$q(u+v)=q(u)+q(v)+\langle u,v\rangle.$$

Fixing a symplectic basis
$(e_1,\ldots,e_g,f_1,\ldots,f_g)$ of $(V,\langle,\rangle)$, 
we define
the \emph{Arf invariant} $a(q)$ of a quadratic form $q$ by
$$a(q) = \sum_{i=1}^g q(e_i)q(f_i).$$
This invariant does not depend on the choice of the symplectic
basis. One says that the form is even (resp. odd) if $a(q)=0$
(resp. $a(q)=1$). The symplectic group
$\Gamma=\Sp(V,\langle,\rangle) \simeq \Sp_{2g}(\FF_2)$ acts transitively on the sets of even
and odd quadratic forms by $(\sigma \cdot q)(v)=q(\sigma^{-1} v)$. 
There are $2^{g-1}(2^g+1)$ (resp. $2^{g-1}(2^g-1)$) forms with Arf
invariants $0$ (resp. $1$).\\
The set $QV$ of quadratic forms on $V$ is a principal homogeneous
space for $V$: if $q \in QV$ and $v \in V$, we define $q+v$ by
$(q+v)(u)=q(u)+\langle v,u\rangle$. Similarly if $q$ and $q'$ are two
quadratic forms, then we can define $v=q+q' \in V$ as the unique
vector such that $\langle v,u \rangle= q(u)+q'(u), \quad \forall u \in
V$. \\

For any quadratic form $q$ we compute $q(w)$ in terms of the
coordinates,
$$w=\lambda_1e_1+\cdots+\lambda_ge_g+\mu_1f_1+\cdots+\mu_gf_g$$ of any
vector $w\in V$. For simplicity, we shall write $w=(\lambda,\mu)$,
with
$\lambda=(\lambda_1,\dots,\lambda_g)$ and $\mu=(\mu_1,\dots,\mu_g)$ in $\FF_2^g$.
In coordinates, the most simple quadratic form is:
\begin{equation} \label{eq:q0}
q_0(w)=\lambda\cdot\mu,
\end{equation}
where $\cdot$ denotes the usual dot product of $g$-tuples. Now,
for any other vector $v\in V$, with coordinates $v=(\eps',\eps)$ (in
this order), the form $q=q_0+v$ acts by:
\begin{equation*}\label{qv}
 q(w)=\eps\cdot
\lambda+\eps'\cdot\mu+\lambda\cdot\mu.
\end{equation*}
We write $q=\Ch{\varepsilon}{\varepsilon'}$. Note that
$$\eps=(q(e_1),\dots,q(e_g)), \quad \eps'=(q(f_1),\dots,q(f_g))$$
and therefore $a(q) = \eps \cdot \eps'$.
In coordinates we have:
\begin{equation*}
\begin{array}{l}
\Ch{\eps}{\eps'} + (\lambda,\mu)=\Ch{\eps+\mu}{\eps'+\lambda},\\
\Ch{\eps_1}{\eps'_1}+\Ch{\eps_2}{\eps'_2}+
\Ch{\eps_3}{\eps'_3}=\Ch{\eps_1+\eps_2+\eps_3}
{\eps'_1+\eps'_2+\eps'_3}.
\end{array}
\end{equation*}
With respect to the symplectic basis
$(e_i,f_j)$, the action of any $\sigma \in \Gamma$ is represented by a
matrix
$\sigma=\begin{pmatrix} a & b  \\ c & d \end{pmatrix}$ with $a,b,c,d$ $g
\times g$-matrices such that ${^t a} d + {^t c}
  b=\textrm{id}$ and $^t a c$ and $^t b d$ are symmetric. Then 
$\sigma \cdot \Ch{\varepsilon}{\varepsilon'} = \Ch{\nu}{\nu'}$
where $$\begin{pmatrix} ^t \nu \\ ^t \nu' \end{pmatrix}
= \begin{pmatrix} d & c  \\ b & a \end{pmatrix} 
\begin{pmatrix} ^t \varepsilon \\ ^t \varepsilon' \end{pmatrix}
+\begin{pmatrix} (c ^t d)_0  \\ (a ^t b)_0 \end{pmatrix} $$
and the $0$ subscript means the column vector  of the diagonal elements of
the matrix. \\

 The following lemma will be useful in
computations and can be easily proven using a basis as above.
\begin{lemma} \label{lem:comp}
Let $q,q',q''$ be three quadratic forms. Then 
$$a(q+q'+q'')= a(q)+a(q')+a(q'') + \langle q+q',q+q''\rangle$$
and
$$q(q'+q'') = a(q+q'+q'')+a(q).$$
\end{lemma}

\begin{definition} Let $S=\{q_1,\ldots,q_{2g+1}\}$ be a set of
  quadratic forms such that any quadratic form $q$ can be written
  $q=\sum \alpha_i q_i \in QV$ with $\alpha_i \in \{0,1\} \subset
  \ZZ$. 
 One says that $S$ is an
\emph{Aronhold set} provided that the Arf invariant of any element
satisfies $$a(q)=\frac{\#q -1}{2} + \begin{cases} 0 &
  g \equiv 0,1 \pmod{4}, \\ 1 & g \equiv 2,3 \pmod{4} \end{cases}$$
where $\#q$ is the odd integer $\sum \alpha_i$.
\end{definition}
There exist Aronhold sets and the
symplectic group $\Gamma$ acts transitively on them.
We call an \emph{Aronhold basis} an ordered Aronhold set and we denote
it $(q_1,\ldots,q_{2g+1})$.
Aronhold bases have a strong connection with the notion of azygetic bases.
\begin{definition} 
An {\emph azygetic family} of vectors is an ordered sequence
$(v_1,\dots,v_k)$ such that $\langle v_i,v_j \rangle=1$ for all $i\ne
j$. An azygetic family of $2g$ vectors is necessarily a basis of $V$;
we say that it is an {\it azygetic basis}. \\
An \emph{azygetic family} of quadratic forms is an ordered sequence
$(q_1,q_2,\ldots,q_{k})$  of quadratic forms, such
that $q_1+q_2,\ldots,q_1+q_{k}$ is an azygetic family of
vectors. It is easy to check that this property is preserved under any
reordering of the family.
\end{definition}

\begin{lemma} \label{lem:aroazy}
If $\{q_1,\,\dots,\,q_{2g+1}\}$ is an Aronhold set, then
$(q_1,\ldots,q_{2g+1})$ is an azygetic family.
\end{lemma}

\begin{proof}
It suffices to check that any triple $q_i,q_j,q_k$ of pair-wise
different quadratic forms is azygetic.
Since we have an Aronhold set the Arf invariants of $q_i+q_j+q_k$
and of $q_i$ are different, i.e. $a(q_i+q_j+q_k)+a(q_i)=1$. Hence
using Lemma \ref{lem:comp}
$$
1=a(q_i+q_j+q_k)+a(q_i)=a(q_j)+a(q_k)+\langle v,w \rangle=\langle v,w \rangle,
$$
where $v=q_i+q_j$ and $w=q_i+q_k$.
\end{proof}
This shows that one can
associate to an Aronhold basis  $(q_1,\ldots,q_{2g+1})$  an azygetic basis  $(q_{2g+1}+q_1,\dots,q_{2g+1}+q_{2g})$. 
This induces a bijection between Aronhold bases and azygetic bases.

\begin{definition}
A fundamental system is an azygetic  family $(q_1,\dots,q_{2g+2})$ of $2g+2$ quadratic forms such that
$q_1,\dots,q_g$ are odd, $q_{g+1},\dots,q_{2g+2}$ are even.
\end{definition}

Let us show how to construct a fundamental system from an Aronhold
basis when $g \equiv 3 \pmod{4}$. 
\begin{proposition} \label{prop:connection}
Let $g \equiv 3 \pmod{4}$,  $S=(q_1,\ldots,q_{2g+1})$ be an Aronhold
basis and denote  $q_S=\sum_{i=1}^{2g+1} q_i$. Let
$v=\sum_{i=g+1}^{2g+1} q_i$, 
then 
$$(p_1,\ldots,p_{2g+2})= (q_1,\ldots,q_g, q_{g+1}+v, \ldots,
q_{2g+1}+v, q_S)$$
is a fundamental system.
\end{proposition}
\begin{proof}
Since $g \equiv 3 \pmod{4}$, it is easy to check that the $g$ first
quadratic forms are odd and the $g+2$ others are even using the
formula  $a(q)=\frac{\#q -1}{2} +  1$ if $q=\sum \alpha_i
q_i$. So it remains to show the azygetic property. Let us denote
$v_1,v_2 \in \{v,0\}$. Clearly,
\begin{eqnarray*}
\langle p_{2g+2} + p_i,p_{2g+2}+p_j
\rangle &=& \langle q_S + q_i+v_1, q_S+q_j+v_2 \rangle \\
&=&  \langle q_S + q_i, q_S+q_j\rangle +  \langle q_S + q_i, v_2
\rangle +  \langle v_1, q_S+q_j \rangle + \langle v_1, v_2 \rangle. 
\end{eqnarray*} 
Since $\# (q_S+q_i+q_j)=2g-1$ we have $a(q_S+q_i+q_j)=1$; hence the
first term is $1$ by Lemma \ref{lem:comp}. The last term is always $0$. The second
term is $1$ if and only if $i \in \{g+1,\ldots,2g+1\}$ (and then
$v_1=v$) and $v_2=v$ (and then $j \in \{g+1,\ldots,2g+1\}$). This is
symmetric in $i$ and $j$. Hence
the two central terms are always equal.
\end{proof}

Finally, we will need  the following operation on fundamental
systems for our proof of Weber's formula. Let $P=(p_1,\ldots,p_{2g+2})$ be a fundamental system. For $1
\leq i \leq g$, let $v_i=p_i+p_{2g+2}$. We will denote $v_i+P$ the
sequence of forms where we exchange the place of the $i$th and
last element in the sequence $(p_1+v_i,\ldots,p_{2g+2}+v_i)$.
\begin{proposition} \label{prop:shift}
With the notation as above, $v_i+P$ is a fundamental system.
\end{proposition}
\begin{proof}
Let us denote $v_i+P=(p_1',\ldots,p_{2g+2}')$. First $p'_i=p_i$ and $p_{2g+2}'=p_{2g+2}$, so we need to check the Arf
invariant of the other elements
\begin{eqnarray*}
a(p_j')& =& a(p_j+ p_i + p_{2g+2}) = a(p_j)+a(p_i)+a(p_{2g+2}) +
\langle p_{2g+2} +p_i , p_{2g+2} + p_j \rangle \\
&=& a(p_j)+a(p_i) +1=a(p_j).
\end{eqnarray*}
Finally let us check the azygetic condition for all triples
$p'_{2g+2},p'_j,p'_k$. For $j,k \ne i$ we have
\begin{eqnarray*}
\langle p_{2g+2}' + p'_j, p_{2g+2}'+p_k' \rangle &=& \langle p_i +
p_j, p_i+p_k  \rangle = 1
\end{eqnarray*}
and 
$$\langle p_{2g+2}' + p'_i, p_{2g+2}'+p_k' \rangle = \langle p_{2g+2}
+ p_i, p_i+p_k \rangle=1,$$
because the fundamental system $P$ is an azygetic family.
\end{proof}

\subsection{Riemann-Jacobi Formula} \label{RJformula}
For $g \geq 1$, let $\HH_{g}=\{\tau \in
  \textrm{GL}_{g}(\CC),  \; ^{t}{\tau} = \tau, \ \Im \tau > 0\}$ be the
\emph{Siegel upper half plane}. For any $x \in \CC$, let $\ex(x)=\exp(2 i \pi x)$.
\begin{definition}
For $\tau \in \HH_g$, $z=(z_1,\ldots,z_g) \in \CC^g$ and 
$$
[q]  = \Ch{\varepsilon}{\varepsilon'} \in \ZZ^g \oplus \ZZ^g,
$$
the function 
$$
\vartheta[q](z,\tau) = \sum_{n \in \ZZ^g}
\ex\left(\frac{1}{2} (n + \varepsilon/2) \tau {^t (n + \varepsilon/2)} + 
(n + \varepsilon/2){^t (z+\varepsilon'/2)}\right).
$$
is well defined and is called the \emph{theta function with characteristic
$[q]$}.
\end{definition}
 Using the notation of Section \ref{rem:iden}, we can identity a
characteristic $[q]$ modulo $2$ with a quadratic form over
$\FF_2^{2g}$, which we will still denote $q$. The form corresponding
to the characteristic $\Ch{0}{0}$ will be denote $q_0$ in the sequel. 
If starting with a quadratic form $q$ (and a fixed symplectic basis),
and if not mentioned otherwise, we choose for the characteristic $[q]$
a specific
representative with coefficients in $\{0,1\}$. The
choice of a representative has an impact on the sign of the theta function.
\begin{lemma}[{\cite[Th.I.3]{farkas}}] \label{lem:signchange}
For any characteristic $\Ch{\eps}{\eps'}$ and $m,n \in \ZZ^g$, one has
$$\vartheta \Ch{\eps + 2m}{\eps' + 2 n }(z,\tau) = (-1)^{n \cdot \eps}
\cdot  \vartheta \Ch{\eps}{\eps'}(z,\tau).$$
\end{lemma}
The  function $z \mapsto \vartheta[q](z,\tau)$ is even (resp. odd) if
$a(q) \equiv \varepsilon_1 {^t \varepsilon_2} \pmod{2} =0$ (resp. $a(q)= 1$). When the function is even, its
value at $z=0$ is called a \emph{Thetanullwert} (with characteristic
$[q]$) and denoted $\vartheta[q](\tau)$. 

\begin{definition} \label{def:jacobian}
Let $[q_1],\ldots,[q_g]$ be $g$ odd characteristics. We denote
$$[q_1,\ldots,q_g](\tau) =\pi^{-g} \cdot  \det \left(\frac{\partial
    \vartheta[q_j](z,\tau)}{\partial z_i}(0,\tau) \right)_{1 \leq i,j
  \leq g}$$
the \emph{Jacobian Nullwert with characteristics $[q_1],\ldots,[q_g]$}.
\end{definition}

There is a vast literature devoted to relations between
Thetanullwerte and Jacobian Nullwerte, originating in the famous
\emph{Jacobi identity}
$$\vartheta\Ch{1}{1}'(0,\tau) = -\pi \cdot \vartheta\Ch{0}{0}(0,\tau) \cdot
\vartheta\Ch{1}{0}(0,\tau) \cdot \vartheta\Ch{0}{1}(0,\tau).$$
The formula has been generalized by Rosenhain, Frobenius, Weber and Riemann
(see \cite{igusahistory} for precise references) up to genus $4$ and
in modern time by Fay \cite{fay} for  genus $5$ (see also
\cite{grushevsky} for  higher derivative relations). Fay also proved that
the Ancients' conjectural formula does not hold for genus $6$. All these results
 fit in the following general background. 
\begin{theorem}[{\cite[Th.3]{igusajacobi},\cite[p.171]{igusahistory},\cite{fay}}]
Let $q_1,\ldots,q_g$ be $g$ odd characteristics  such that
the function  $[q_1,\ldots,q_g](\tau)$ is different from the
constant $0$ and is contained in the $\CC$-algebra generated by the
Thetanullwerte. Then 
$$[q_1,\ldots,q_g](\tau)= \sum_{\{q_{g+1},\ldots,q_{2g+2}\} \in \calS}
\pm \prod_{i=g+1}^{2g+2} \vartheta[q_i](\tau),$$
 where $\calS$ is the set of all sets of $g+2$
 even forms  $\{q_{g+1},\ldots,q_{2g+2}\}$ such
 that $(q_1,\ldots,q_{2g+2})$  is a fundamental
 system. The signs are independent of  $\tau$.

\end{theorem}

For $g=3$, the result can  be stated in the following simpler
form.
\begin{corollary} \label{cor:jacobi}
Let $(q_1,\ldots,q_{8})$ be a fundamental system, then 
$$[q_1,q_2,q_3](\tau) = \pm   \prod_{i=4}^8 \vartheta[q_i](\tau),$$
and the sign does not depend on $\tau$.
\end{corollary}
 The sign can actually be determined by computing with  a well chosen  fundamental
system and with a scalar matrix $\tau$ in order to reduce the problem
to a (non-zero) Jacobi identity. One then moves to a different
fundamental system by the transitive action of $\Gamma$ (see Section \ref{sec:sign}).

\subsection{Link between the curve and its Jacobian} \label{sub:link}
We follow here the presentation of \cite{guardia}.
Let $\calC$ be a smooth, irreducible projective curve of genus $g>0$ over
$\C$ and $\boldsymbol{\omega}=(\omega_1,\ldots,\omega_g)$ be a basis of regular
differentials. Let $\boldsymbol{\delta}=(\delta_1,\ldots,\delta_{2g})$ be a symplectic basis of
$H_1(\calC,\Z)$ such that the intersection pairing has matrix
$\begin{bmatrix} 0 & \textrm{id} \\   -\textrm{id} & 0 \end{bmatrix}.$
With respect to these choices, the period matrix of $\calC$ is
$\Omega=[\Omega_1,\Omega_2]$ where $\Omega_1=(\int_{\delta_i} \omega_j)_{1
  \leq i \leq g, 1
  \leq j \leq g}$ and
$\Omega_2=(\int_{\delta_i} \omega_j)_{ g+1 \leq i \leq 2g,1 \leq j
  \leq g,}$. We consider a second basis $\boldsymbol{\eta}$  of
regular differentials  obtained by $\boldsymbol{\eta}=
\Omega_1^{-1} \boldsymbol{\omega}$. The period matrix with respect to this new
basis is $[\textrm{id}, \tau]$ where $\tau=\Omega_1^{-1} \Omega_2 \in
\mathbb{H}_g$ and we let $\Jac(\calC) = \CC^g/(\ZZ^g + \tau \ZZ^g)$.\\
Let us denote for $1 \leq i \leq g$,
$$e_i = {\left(\frac{1}{2} \int_{\delta_i} \eta_j \right)_{1\leq j
    \leq g}} =  (0,\ldots,0,\frac{1}{2},0,\ldots,0) \in \CC^g, \quad 
f_i = { \left(\frac{1}{2} \int_{\delta_{g+i}} \eta_j\right)_{1\leq j
    \leq g}} \in \CC^g$$
and $v=\sum_{i=1}^g \lambda_i e_i + \mu_j f_j = (\lambda,\mu)$
with $\lambda,\mu \in\ZZ^g$. We let $W$ be the $\ZZ$-module generated
by $e_1,\ldots,e_g,f_1,\ldots,f_g$, so that $\Jac(\calC)[2] = W/\ZZ^g+\tau \ZZ^g$.
An element $v\in W$ also acts on a theta function as follows.
\begin{lemma}[{\cite[Th.I.5]{farkas}}]
Let $[q]=\Ch{\epsilon}{\epsilon'}$ be a characteristic and
$v=(\lambda,\mu) \in W$.Then
\begin{equation} \label{eq:theta}
\vartheta[q](z+v,\tau)=\ex\left(-\frac{1}{4} \mu {^t (\epsilon'+\lambda)} -
  \frac{1}{2} \mu
{^t z} - \frac{1}{8} \mu \tau {^t\mu}\right) \cdot \vartheta \car{\epsilon +
  \mu}{ \epsilon' + \lambda}(z,\tau).
\end{equation}
\end{lemma}
We will write $[q]+v = \car{\epsilon +
  \mu}{ \epsilon' + \lambda}$ (the convention is different from
\cite[Def.I.6]{farkas}). Using this notation, we can see the difference
of two characteristics as an element of $W$. \\

Thanks to the identifications of
Section \ref{rem:iden}, the reduction modulo $2$ of the characteristics
and of $(\lambda,\mu)$ is coherent with the theory of quadratic forms
on the $\FF_2$-vector space $V=\Jac(\calC)[2]$, naturally equipped with the
Weil pairing and for the choice of  the symplectic basis
induced by the $e_i,f_i$ on $V$. If we denote $\bar{v}  \in V$ the
class of $v$, $\bar{v}$ is 
identified with $(\lambda \pmod{2}, \mu \pmod{2})$ in the isomorphism
$V \simeq \FF_2^{2g}$ and we see
that $q+\bar{v}$ is the quadratic form associated to the characteristic  $[q]+v$.\\ 

Let $\Theta \subset \Jac(\calC)$ be the zero divisor of the theta
function $\vartheta(z,\tau)$.  The divisor $\Theta$ can be interpreted in terms of the geometry
of $\calC$. For a divisor $D \in \Pic(\calC)$, we denote $l(D)$ the dimension
of the Riemann-Roch space associated to $D$.

\begin{proposition}[Riemann theorem] \label{prop:rie}
Let $W_{g-1}=\{ D \in \Pic^{g-1}(\calC), l(D)>0\}$ and $\kappa$ the
canonical divisor on $\calC$. 
There exists a unique divisor class $D_0$ of degree $g-1$ with $2 D_0 \sim
\kappa$ and $l(D_0)$ even
such that
$W_{g-1} = \Theta + D_0.$
Moreover for any $v \in V$, $\textrm{mult}_{v}(\Theta) = l(D_0+v)$.
\end{proposition}
A divisor (class) $D$ such that $2D \sim \kappa$  is called a \emph{theta
  characteristic divisor}. Any theta characteristic divisor $D$ is
linearly equivalent to $D_0+v$ with $v =(\lambda,\mu) \in V$.  We can associate to
$D$ the quadratic form $q=q_0+v$ where $q_0$ is the quadratic form
defined in \eqref{eq:q0}. Note that
$$a(q)= a(q_0+v) \equiv
\textrm{mult}_{v}(\Theta) \pmod{2}$$ since $\textrm{mult}_{v}(\Theta)$
 is equal to the multiplicity at $0$ of  $\vartheta[q](z,\tau)$ and
 the latter has the same parity as $q$.  Therefore, using Proposition \ref{prop:rie}, for any $w \in V$, one has
$$q(w)=a(q+w)+a(q) \equiv l(D+w) + l(D) \pmod{2}.$$
\begin{lemma} \label{rem:char}
Any theta characteristic divisor $D$ corresponds to a quadratic form $q$
defined by
$$q(v) =l(D+v)+l(D) \pmod{2}, \quad v \in V.$$
It has Arf invariant   $a(q) \equiv l(D) \pmod{2}$. Note that 
 the divisor $D_0$ corresponds to the quadratic
form $q_0$.
\end{lemma}
Conversely, starting from a quadratic form $q$, this correspondence defines a
divisor class $D_q =D_0 + q_0+ q$.\\
The basis of regular differentials $\boldsymbol{\omega} $ defines the canonical map
\begin{eqnarray*}
\phi  :& \calC &\to \P^{g-1} \\
& P &\mapsto (\omega_1(P) : \ldots : \omega_g(P)).
\end{eqnarray*}
If $D \in Pic^{g-1}(\calC)$ is such that $l(D)=1$, then $D \sim P_1 + \ldots P_{g-1}$
with $\phi(P_i) \in \phi(\calC)$ being the support of the intersection of
$\phi(\calC)$ with a unique hyperplane $H_D$ of $\PP^{n-1}$. An
equation of this hyperplane is given by the following proposition.

\begin{proposition}[{\cite{guardia}}] \label{hyperplane}
Let us denote $\vartheta_i(z,\tau)=\frac{\partial\vartheta}{\partial z_i}(z,\tau)$.
Let $D \in \Pic^{g-1}(\calC)$ such that $l(D)=1$ then
$$\left(\vartheta_1(D-D_0,\tau),\cdots,\vartheta_g(D-D_0,\tau)\right)
\Omega_1^{-1} \left(\begin{array}{c} X_1 \\ \vdots \\ X_{g}
  \end{array}\right)=0$$ is an equation of $H_D$.
\end{proposition}
Let $q_1,\ldots,q_g$ be $g$ odd quadratic forms and assume that the
theta characteristic divisors $D_{q_i}$ are such that
$l(D_{q_i})=1$. Then $H_{D_{q_i}}$ is tangent to the curve at each
point $\phi(P_i)$ such that $D_{q_i} \sim P_1 + \ldots + P_{g-1}$.
Let $\beta_{{q_i}} \in
\CC[X_1,\ldots,X_{g}]$ be any linear polynomials such that
$H_{D_{q_i}}$ is the hyperplane with equation $\beta_{{q_i}} =0$.

\begin{corollary} \label{link}
With the notation above, there exist constants $\eta_i=\eta_{[q_i],\beta_{q_i}}$
depending on $[q_i]$, $\beta_{q_i}$ (and the period matrix $\Omega$) such that $$[\beta_{q_1},\ldots, \beta_{q_{g}}]
= \left(\prod_{i=1}^g \eta_i\right) \cdot [q_1,\ldots,q_g]$$ where $[\beta_{q_1},\ldots, \beta_{q_{g}}]$ is the
determinant of the coefficients of the $\beta_{q_i}$ in the basis $X_1,\ldots,X_g$.  
\end{corollary}
\begin{proof}
Let $v_i=D_{q_i}-D_0=[q_0]+[q_i]=(\lambda_i,\mu_i) \in W$ for $1 \leq
i \leq g$. By \eqref{eq:theta} one has
$$\vartheta(z +v_i,\tau) = \ex\left(-\frac{1}{4} \mu_i {^t \lambda_i} -
 \frac{1}{2} \mu_i {^t z} - \frac{1}{8} \mu_i \tau  {^t \mu_i}\right)
\cdot \vartheta \Ch{\mu_i}{\lambda_i}(z,\tau)$$
and  for all $1 \leq j \leq g$, we have
$$\vartheta_j(D_{q_i}-D_0,\tau) = \frac{\partial \vartheta(z)}{\partial z_j}(v_i,\tau) = \frac{\partial \vartheta(z+v_i)}{\partial z_j}(0,\tau) = c_i
\cdot \frac{\partial \vartheta[q_i](z,\tau)}{\partial z_j}(0,\tau)$$
where $c_i$ depends on $\tau$ and $[q_i]$.
Proposition \ref{hyperplane} shows that
$$\beta_{q_i} = c_i' \cdot
\left(\frac{\partial \vartheta[q_i](z,\tau)}{\partial z_1}(0,\tau), \cdots, \frac{\partial \vartheta[q_i](z,\tau)}{\partial z_g}(0,\tau)\right)
\Omega_1^{-1} \left(\begin{array}{c} X_1 \\ \vdots \\ X_{g}
  \end{array}\right)$$ 
for a constant $c_i'$ depending on $\beta_{q_i},[q_i]$ and $\tau$. Taking the
determinant, we get the result.
\end{proof}

\section{Proofs of Weber's formula} \label{sec:proof-weber}

We now restrict  to $g=3$ and we assume that $\calC$ is a non
hyperelliptic curve of genus $3$ over $\C$. Let
$(\omega_1,\omega_2,\omega_3)$ be a basis of regular differentials. The
canonical embedding $\phi : P \mapsto (\omega_1(P) : \omega_2(P) :
\omega_3(P)) \in \PP^2$  identifies $\calC$ with a smooth plane
quartic. Let $D$ be a theta characteristic divisor of $\calC$. If
$l(D)>0$, then $D \sim P+Q$, where $P,Q \in \calC$. But then $l(D)=1$,
otherwise, there would be a non constant function of degree $2$ with
pole $P+Q$ and $\calC$ would be hyperelliptic.
For the canonical
embedding, the line $H_D$ defined by $P,Q$ (resp. the tangent to $\calC$ if $P=Q$) is
tangent to $\calC$ at $P$ and $Q$ (resp. has intersection multiplicity $4$
at $P$). Such a line is called a
\emph{bitangent} to $\calC$. Using the bijection of Lemma \ref{rem:char},
we see that such a $D$ correspond to an odd quadratic form $q$. Hence
the number of bitangents in $28$. To describe this set, we introduce
an Aronhold set $S=\{q_1,\ldots,q_7\}$ associated to a given even form
$q_S=\sum_{i=1}^7 q_i$ (this is always possible by the transitive
action of $\Gamma$ on Aronhold sets). For all $1 \leq  i \ne j \leq
7$, we denote $q_{ij}=q_S+q_i+q_j$ the sum of $5$ distinct
$q_i$s, hence this is an odd form. The $28$ odd forms can all be
written as $q_i$ or $q_{ij}$ 
and we denote by $D_i=D_{q_i}$ or $D_{ij}=D_{q_{ij}}$ (resp. $\beta_i$,
$\beta_{ij}$) the theta characteristic divisor (resp. an arbitrary
fixed linear polynomial defining $H_{D_{q_i}}$ or $H_{D_{q_{ij}}}$)
associated to them. Note also that any even form different from $q_S$
can be written $q_{ijk}=q_i+q_j+q_k$ with $i,j,k$ distinct.
We can now state Weber's formula.

\begin{theorem}[{Weber's formula \cite[p.162]{weber}}] \label{eq:jacobi}
Let $q_S,q_T$ be two distinct even forms. Let $S=\{q_1,\ldots,q_7\}$ be
an Aronhold set such that $q_S=\sum_{i=1}^7 q_i$ and  assume that we have
ordered $S$ so that $q_1+q_2+q_3=q_T$. Define a Riemann matrix
$\tau \in \HH_3$ attached to $\Jac(\calC)$ following the beginning of
Section \ref{sub:link}. Then 
\begin{equation}
\left(\frac{\vartheta[q_S](\tau)}{\vartheta[q_T](\tau)}\right)^4=
(-1)^{a(q_0+q_S+q_T)} \cdot 
\frac{\left[\beta_1,\beta_2,\beta_3\right] \cdot
\left[\beta_1,\beta_{12},\beta_{13}\right]
 \cdot \left[\beta_{12},\beta_{2},\beta_{23}\right] \cdot
\left[\beta_{13},\beta_{23},\beta_{3}\right]}{\left[
\beta_{23},\beta_{13},\beta_{12}\right] \cdot
\left[\beta_{23},\beta_{3},\beta_{2}\right] \cdot
\left[\beta_{3},\beta_{13},\beta_{1}\right] \cdot
\left[\beta_{2},\beta_{1},\beta_{12}\right]}
\end{equation}
where $[\beta_i, \beta_{j},\beta_k]$ is the determinant of the
coefficients of $\beta_i,\beta_j$ and $\beta_k$. 
\end{theorem}
Let us point out that each defining polynomial of a bitangent appears
as many times on the
numerator as on the denominator, so the quotient of the two
expression does not depend on the choice of a fixed polynomial. 
Similarly, as the characteristics $[q_S],[q_T]$ appear in
Thetanullwerte raised to an even power, one can choose any representative
for the characteristics associated to $q_S,q_T$. However, the dependence on the choices of
symplectic basis and regular differentials appear on the left in the
choice of $\tau$ and on the right side in the choice of $q_0$.

\subsection{Sketch of  Weber's proof} \label{sec:weber}
The original proof of Weber's formula can be found in his book
\cite{weber}. We want to give here an overview of his proof,
formulated in a simpler and modern form. For symmetry, we denote
$p_1=q_S$ and $p_2=q_T$ and then 
$$p_1+p_2=q_1+q_{23}=q_2+q_{13}=q_3+q_{12}.$$
 Let
$$D_{1} \sim A+B, \quad D_{23} \sim G+H$$
be the two theta characteristics divisors associated to $q_1$ and
$q_{23}$. The points $A,B$ (resp. $H,G$) are then the support of
the bitangents $\beta_1$, (resp. $\beta_{23}$). Let $S=S_1+S_2+S_3$ be an arbitrary generic
effective divisor of degree $3$ on $\calC$.  
We introduce
$$f_{i,S}(P) = \vartheta[p_i](P+S-\kappa)$$ with $\kappa = 2(A+B)$, so
this fixes a precise value for $f_{i,S}(P)$ in $\CC$ once paths have been chosen
to each point. The $f_{i,S}(P)$ are regular sections of line bundles
over $\calC$. According to Riemann theorem \cite[V.Th.1]{farkas}, if
$f_{i,S}$ is not identically zero then its zero divisor $(f_{i,S})_0$
has degree three and satisfies
$$(f_{i,S})_0 \sim D_0+ (p_i + q_0) + \kappa-S = D_{p_i}
+\kappa -S.$$
Since $l(\kappa + D_{p_i})=4$, we let $t_i,u_i,v_i,w_i$ be a basis
of sections (called \emph{Wurzelfunctionen} in Weber's book). We then define
\begin{equation} \label{det}
\chi_{i,S}(P) = \det \begin{pmatrix} t_i(P) & u_i(P) & v_i(P) &
  w_i(P) \\
t_i(S_1) & u_i(S_1) & v_i(S_1) &
  w_i(S_1) \\
t_i(S_2) & u_i(S_2) & v_i(S_2) &
  w_i(S_2) \\
t_i(S_3) & u_i(S_3) & v_i(S_3) &
  w_i(S_3) \end{pmatrix}.
\end{equation}
Since $\chi_{i,S}(S_j)=0$ for $1
\leq j \leq 3$, we see that $(\chi_{1,S})_0 = S + R_i$ where $R_i$
is an effective divisor
of degree $3$, uniquely defined by $R_i+S \sim \kappa + D_{p_i}$.
Now $$(f_{i,S})_0 \sim D_{p_i}
+\kappa -S \sim R_i$$
so actually $(f_{i,S})_0=R_i$. Therefore, $(f_{1,S})_0 - (f_{2,S})_0 =
R_1-R_2= (\chi_{1,S})_0 -
(\chi_{2,S})_0$ and there exists a constant $\alpha_S$ such that 
$$\frac{f_{1,S}(P)}{f_{2,S}(P)} = \alpha_S \cdot \frac{\chi_{1,S}(P)}{\chi_{2,S}(P)}.$$  
\begin{lemma}
$\alpha_S$ does not depend on $S$.
\end{lemma}
\begin{proof}
One has 
$$\frac{f_{1,S}(P)}{f_{2,S}(P)} \cdot  \frac{\chi_{2,S}(P)}{\chi_{1,S}(P)}= \alpha_S.$$
We have to prove that the expression on the left side does not depend
on the support of $S=S_1+S_2+S_3$. It is enough to show that
$\alpha_S=\alpha_{S_1'+S_2+S_3}$ for another generic point $S_1'$. Note that
$f_{i,S}(S_1') = \vartheta[p_i](S_1'+S_1+S_2+S_3-\kappa)
=f_{i,S_1'+S_2+S_3}(S_1)$
and  $\chi_{i,S}(S'_1) =- \chi_{i,S'_1+S_2+S_3}(S_1)$. Hence
$$\alpha_S=\frac{f_{1,S}(S_1')}{f_{2,S}(S_1')} \cdot  \frac{\chi_{2,S}(S_1')}{\chi_{1,S}(S_1')}=\frac{f_{1,S'_1+S_2+S_3}(S_1)}{f_{2,S_1'+S_2+S_3}(S_1)} \cdot  \frac{\chi_{2,S_1'+S_2+S_3}(S_1)}{\chi_{1,S_1'+S_2+S_3}(S_1)}=\alpha_{S_1'+S_2+S_3}.$$
\end{proof}
In the sequel we are going to use two particular divisors $S$.
\begin{lemma}
If $S = B+A+B$ then
$$\frac{f_{1,S}(A)^2}{f_{2,S}(A)^2} = \frac{\vartheta[p_1](0)^2}{\vartheta[p_2](0)^2}.$$
If moreover $S'=B+ G+H$ then 
$$\frac{f_{1,S'}(P)^2}{f_{2,S'}(P)^2} = (-1)^{a(q_0+p_1+p_2)}
\cdot \frac{f_{2,S}(P)^2}{f_{1,S}(P)^2}.$$
\end{lemma}
\begin{proof}
The first equality is trivial. As for the second, let $[p_1] =
\Ch{\eps}{\eps'}$ and $$(G+H)-(A+B) \sim D_{23}-D_1 =[q_{23}]-[q_1]= (\lambda,\mu),$$  so that
$[p_2]=[p_1]+[q_{23}]-[q_1]=\Ch{\epsilon + \mu}{\epsilon' + \lambda}$ (the choices for
the lifts of the quadratic forms are irrelevant because we are going to take squares).
Then using \eqref{eq:theta}
\begin{eqnarray*}
f_{1,S'}(P)^2 &= & \vartheta[p_1](P + B+G+H-\kappa)^2 \\
&=& 
\vartheta[p_1](P + B +A +B -\kappa + (G+H)-(A+B))^2\\
& =& (-1)^{\mu \cdot (\eps'+\lambda)}
\cdot c_{\tau,\mu,z}  \cdot f_{2,S}(P)^2
\end{eqnarray*}
where $z =P + B +A +B -\kappa$, $c_{\tau,\mu,z}$ is a constant
depending on $\tau,\mu, z$ and $$f_{2,S'}(P)^2 = (-1)^{\mu \cdot \eps'}
\cdot c_{\tau,\mu,z}  \cdot f_{1,S}(P)^2.$$
Hence for the quotient we get
$$\frac{f_{1,S'}(P)^2}{f_{2,S'}(P)^2} = (-1)^{\mu \cdot \lambda}
\cdot \frac{f_{2,S}(P)^2}{f_{1,S}(P)^2}.$$
\end{proof}
From this we get that
$$\frac{f_{1,S}(A)^2 \cdot f_{2,S'}(A)^2}{f_{2,S}(A)^2 \cdot
  f_{1,S'}(A)^2} = (-1)^{a(q_0+p_1+p_2)} \cdot  \frac{\vartheta[p_1](0)^4}{\vartheta[p_2](0)^4}
= \frac{\chi_{1,S}(A)^2 \cdot
  \chi_{2,S'}(A)^2}{\chi_{2,S}(A)^2 \cdot \chi_{1,S'}(A)^2}.$$
Note, however, that the expression 
$\chi_{1,S}(A)/\chi_{2,S}(A)$ take the indeterminate form $0/0$ so we need first to resolve this
ambiguity and then we will express everything in terms of the bitangents.\\

We denote as Weber did
$\sqrt{\beta_i}$ (resp. $\sqrt{\beta_{ij}}$) a (fixed) section (\emph{Abelsche Function}) of the
bundle associate to $D_i$ (resp. to $D_{ij}$).
Since
$$p_1+q_1 = q_3+ q_{13} = q_2 + q_{12}, \quad p_1+q_{23} =
q_2+q_3=q_{13}+q_{12}$$
and
$$p_2+q_1 = q_2+ q_{3} = q_{13} + q_{12}, \quad p_2+q_{23} =
q_3+q_{13}=q_{2}+q_{12}$$ 
 We can then choose for
$t_i,u_i,v_i$ and $w_i$ the following expressions
$$t_1= \sqrt{\beta_1 \beta_3 \beta_{13}}, \; u_1= \sqrt{\beta_1
  \beta_2 \beta_{12}}, \; v_1 = \sqrt{\beta_{23} \beta_2 \beta_{3}}, \;
w_1 = \sqrt{\beta_{23} \beta_{13} \beta_{12}}$$
and
$$t_2= \sqrt{\beta_1 \beta_2 \beta_{3}}, \; u_2= \sqrt{\beta_1
  \beta_{13} \beta_{12}}, \; v_2 = \sqrt{\beta_{23} \beta_{3} \beta_{13}}, \;
w_2 = \sqrt{\beta_{23} \beta_{2} \beta_{12}}.$$
We start with a divisor $S=S_1+A+B$ and we will let $S_1=B$ and $P=A$ once we have resolved the ambiguity
$0/0$. Note that $\sqrt{\beta_1}(A)=\sqrt{\beta_1}(B)=0$. Hence the
determinant \eqref{det} becomes
$$\chi_{i,S}(P) = (t_i(P) u_i(S_1) - t_i(S_1) u_i(P)) \cdot
(v_i(A) w_i(B) - v_i(B) w_i(A)).$$ 
In the quotient
$\chi_{1,S}(P)/\chi_{2,S}(P)$ we see that  $\sqrt{\beta_1}(P)
\sqrt{\beta_1}(S_1)$ and $\sqrt{\beta_{23}}(A)
\sqrt{\beta_{23}}(B)$ appear in the numerator and in the denominator, so after cancellation and taking $S_1=B$ and $P=A$, we are left with
(writting $\sqrt{\beta_i^A}= \sqrt{\beta_i}(A)$ and $\sqrt{\beta_i^B}=\sqrt{\beta_i}(B)$) 

$$\frac{\chi_{1,S}(A)}{\chi_{2,S}(A)} = \frac{\left(\sqrt{\beta_3^{A}
    \beta_{13}^{A} \beta_2^{B} \beta_{12}^{B}} - \sqrt{\beta_3^{B}
    \beta_{13}^{B} \beta_2^{A} \beta_{12}^{A}}\right) \cdot \left(\sqrt{\beta_2^{A}
    \beta_{3}^{A} \beta_{13}^{B} \beta_{12}^{B}} - \sqrt{\beta_2^{B}
    \beta_{3}^{B} \beta_{13}^{A} \beta_{12}^{A}}\right)}{\left(\sqrt{\beta_2^{A}
    \beta_{3}^{A} \beta_{13}^{B} \beta_{12}^{B}} - \sqrt{\beta_2^{B}
    \beta_{3}^{B} \beta_{13}^{A} \beta_{12}^{A}}\right) \cdot \left(\sqrt{\beta_3^{A}
    \beta_{13}^{A} \beta_{2}^{B} \beta_{12}^{B}}) - \sqrt{\beta_3^{B}
    \beta_{13}^{B} \beta_{2}^{A} \beta_{12}^{A}}\right)}=1.$$

\begin{remark} \label{rem:generic}
Until this point, the proof could be easily generalized to  a curve of arbitrary genus $g
\geq 3$. Let us indicate the main modifications. One would consider an effective divisor $S=S_1+\ldots+S_{2g-3}$ of
degree $2g-3$ and the section
$$\chi_{i,S}(P) =\det \begin{pmatrix} t^{(1)}_i(P) & \cdots &
  t^{(2g-2)}_i(P) \\
t^{(1)}_i(S_1) & \cdots &
  t^{(2g-2)}_i(S_1) \\
\vdots & & \vdots \\
t^{(1)}_i(S_{2g-3}) & \cdots &
  t^{(2g-2)}_i(S_{2g-3}) \\
\end{pmatrix}, \quad 1 \leq i \leq 2$$
for the bundle associated to the
divisor $\kappa + D_{p_i}$.\\
The previous decompositions of $p_1+p_2$ as sum of two odd characteristics
are special cases of Steiner systems \cite{dolgacag,nart-un}. 
In general there are
$2^{g-2}(2^{g-1}-1)$ pairs $\{q_i,\bar{q_i}\}$ of odd characteristics
such that $p_1+p_2=q_i+\bar{q_i}$ (above we wrote only half of them).
 Among the characteristics $q_i,\bar{q_i}$ consider the ones which also
 appears in the pairs of the Steiner system relative to
 $p_1+q_1$. After ordering we can write $p_1+q_1 = p_2+\bar{q_1}$ in
 $g+1$ ways $q_i +q_j$ or $\bar{q_i}
 + \bar{q_j}$. One has similarly
$p_1 + \bar{q_1} =  p_2 + q_1$ in $g+1$ ways as $\bar{q_i}+ q_j$ or $q_i + \bar{q_j}$
for the same indices.
If we denote $(i)$ (resp. $(\bar{i})$) a section relative to the
bundle $D_{q_i}$ (resp. $D_{\bar{q_i}}$) we then choose to write for
the $g+1$ choices of $\{i,j\}$ above
$$t_1^{(k)} = (1) (i) (j) \; \textrm{or} \; (1) (\bar{i}) (\bar{j}),
\quad t_2^{(k)} = (\bar{1}) (i) (j) \; \textrm{or} \; (\bar{1}) (\bar{i}) (\bar{j}), \quad 1 \leq
k \leq g-1$$
and
$$t_1^{(k)} = (\bar{1}) (\bar{i}) (j) \; \textrm{or} \; (\bar{1}) (i)
(\bar{j}), \quad t_2^{(k)} = (1) (\bar{i}) (j)  \; \textrm{or} \; (1) (i) (\bar{j}), \quad g \leq
k \leq 2g-2.$$
 The
support of the theta-characteristic divisor $D_{q_1}$ is a
sum of $g-1$ points $A_1,\ldots,A_{g-1}$. Letting first
 $(S_{g-1},\ldots,S_{2g-3})=(A_1,\ldots,A_{g-1})$ gives the sections $\chi_{i,S}(P)$
 as products of determinants of size $g-1$ from which we can simplify
 the sections $(1)$ and $(\bar{1})$ in the quotient
 $\chi_{1,S}(P)/\chi_{2,S}(P)$. 
It is then enough to take $(P,S_1,\ldots, S_{g-2})= (A_1,A_2,\ldots, A_{g-1})$ to obtain the same expression
 for the numerator and denominator and conclude that the quotient is $1$.   
\end{remark}

We now deal with the divisor $S'=B + G+H$. We
now have  $\sqrt{\beta_{23}}(G)=\sqrt{\beta_{23}}(H)=0$; hence
$$\chi_{i,S'}(A) = -  (v_i(A) w_i(B) - v_i(B) w_i(A)) \cdot
(t_i(G) u_i(H) - t_i(H) u_i(G)).$$
Again we can simplify a bit the quotient (writing $\sqrt{\beta_i^G}=
\sqrt{\beta_i}(G)$ and $\sqrt{\beta_i^H}= \sqrt{\beta_i}(H)$)
$$\frac{\chi_{1,S'}(A)}{\chi_{2,S'}(A)} = \frac{\overbrace{\left(\sqrt{\beta_2^{A}
    \beta_{3}^{A} \beta_{13}^{B} \beta_{12}^{B}} - \sqrt{\beta_2^{B}
    \beta_{3}^{B} \beta_{13}^{A} \beta_{12}^{A}}\right)}^{M_1} \cdot \overbrace{\left(\sqrt{\beta_3^{G}
    \beta_{13}^{G} \beta_{2}^{H} \beta_{12}^{H}} - \sqrt{\beta_3^{H}
    \beta_{13}^{H} \beta_{2}^{G} \beta_{12}^{G}}\right)}^{M_2}}{\underbrace{\left(\sqrt{\beta_3^{A}
    \beta_{13}^{A} \beta_{2}^{B} \beta_{12}^{B}} - \sqrt{\beta_3^{B}
    \beta_{13}^{B} \beta_{2}^{A} \beta_{12}^{A}}\right)}_{N_1} \cdot \underbrace{\left(\sqrt{\beta_2^{G}
    \beta_{3}^{G} \beta_{13}^{H} \beta_{12}^{H}}) - \sqrt{\beta_2^{H}
    \beta_{3}^{H} \beta_{13}^{G} \beta_{12}^{G}}\right)}_{N_2}}.
$$
Using the fact that the space of regular sections of the bundle
associated to the divisor $\kappa+(p_1+p_2)$ has
dimension $2$, we see that there is a linear relation of the form
$$h_1 \sqrt{\beta_1 \beta_{23}} + h_2 \sqrt{\beta_2 \beta_{13}} + h_3
\sqrt{\beta_3 \beta_{12}}=0.$$
Changing the value of the $\sqrt{\beta_i}$, we can even assume that
$h_1=h_2=1$ and $h_3=-1$. Using the fact that
$\sqrt{\beta_1^A}=\sqrt{\beta_1^B}=\sqrt{\beta_{23}^G}=\sqrt{\beta_{23}^H}=0$, we get that 
\begin{equation} \label{model}
\sqrt{\beta_2^A \beta_{13}^A} = \sqrt{\beta_3^A \beta_{12}^A}, \quad
\sqrt{\beta_2^B \beta_{13}^B} = \sqrt{\beta_3^B \beta_{12}^B}
\end{equation}
and similarly for $G,H$.
We can now rewrite the $M_i,N_i$ in the following way
$$\sqrt{\beta_3^A \beta_3^B} \cdot M_1= \sqrt{\beta_2^A \beta_2^B}
  \cdot \left(\beta_3^A \beta_{13}^B - \beta_3^B \beta_{13}^A\right), \;
  \sqrt{\beta_3^A \beta_3^B} \cdot N_1 = \sqrt{\beta_{13}^A \beta_{13}^B}
  \cdot \left(\beta_{3}^A \beta_{2}^B - \beta_{3}^B \beta_{2}^A\right),$$

$$\sqrt{\beta_{3}^G \beta_{3}^H} \cdot M_2= \sqrt{\beta_{13}^G \beta_{13}^H}
  \cdot \left(\beta_{3}^G \beta_{2}^H - \beta_{3}^H \beta_{2}^G\right), \;
  \sqrt{\beta_{3}^G \beta_{3}^H} \cdot N_2 = \sqrt{\beta_{2}^G \beta_{2}^H}
  \cdot \left(\beta_{3}^G \beta_{13}^H - \beta_{3}^H \beta_{13}^G\right).$$

Now, we write $\beta_3$ as a linear combinaison of  $\beta_{13},\beta_2,\beta_1$
(resp.  $\beta_{13},\beta_2,\beta_{23}$)
\begin{equation} \label{sys1}
\beta_3 = a_1 \beta_{13} + b_1 \beta_{2} + c_1 \beta_1 = a_2
\beta_{13} + b_2 \beta_{2} + c_2 \beta_{23}. 
\end{equation}
Using the first equality we get
$$\begin{cases}
\beta_3^A &= a_1 \beta_{13}^A + b_1 \beta_{2}^A,  \\
\beta_3^B &= a_1 \beta_{13}^B + b_1 \beta_{2}^B.  
\end{cases}$$
Hence using Cramer's rule we get
$$\frac{M_1}{N_1} = \frac{\sqrt{\beta_2^A
    \beta_2^B}}{\sqrt{\beta_{13}^A \beta_{13}^B}} \cdot
\frac{b_1}{a_1} \;\textrm{and similarly} \; \frac{M_2}{N_2} = \frac{\sqrt{\beta_{13}^G
    \beta_{13}^H}}{\sqrt{\beta_{2}^G \beta_{2}^H}} \cdot
\frac{a_2}{b_2}.$$
It remains to deal with the quotient $\sqrt{\beta_2^A
  \beta_2^B}/\sqrt{\beta_{13}^A\beta_{13}^B}$ (and similarly with $\sqrt{\beta_{13}^G
    \beta_{13}^H}/\sqrt{\beta_{2}^G \beta_{2}^H}$). In order to do so, we
introduce two other linear combinaisons
\begin{equation} \label{sys2}
\beta_{12} = a'_1 \beta_{13} + b'_1 \beta_{2} + c'_1 \beta_1 = a'_2
\beta_{13} + b'_2 \beta_{2} + c'_2 \beta_{23}. 
\end{equation}
Because $\beta_{12}^A \beta_3^A = \beta_{13}^A \beta_2^A$ by \eqref{model},
we can rewrite this equality using \eqref{sys1}
$$ \beta_{13}^A \beta_2^A = \beta_{12}^A \beta_{3}^A =    (a_1' \beta_{13}^A
+ a_2' \beta_2^A) \cdot (a_1 \beta_{13}^A
+ a_2 \beta_2^A). $$ 
Hence $$\frac{\beta_2^A}{\beta_{13}^A} = (a_1 + b_1
\frac{\beta_2^A}{\beta_{13}^A}) \cdot (a_1' + b_1'
\frac{\beta_2^A}{\beta_{13}^A})$$ 
and we get the same expression replacing $A$ by $B$. Therefore, the
quotients $\frac{\beta_2^A}{\beta_{13}^A}$ and
$\frac{\beta_2^B}{\beta_{13}^B}$ can be seen as the two solutions of a
quadratic equation and  their product is equal to the
constant term divided by the leading coefficients; hence
$$\frac{\beta_2^A \beta_2^B}{\beta_{13}^A \beta_{13}^B} = \frac{a_1
  a_1'}{b_1 b_1'}$$
and similarly
$$\frac{\beta_{13}^G \beta_{13}^H}{\beta_{2}^G \beta_{2}^H} =
\frac{b_2 b_2'}{a_2 a_2'}.$$
Putting everything together, we get
\begin{eqnarray*}
\frac{\vartheta[p_1](0)^4}{\vartheta[p_2](0)^4} &=&
(-1)^{a(q_0+p_1+p_2)} \cdot \frac{N_1^2 N_2^2}{M_1^2 M_2^2} = 
(-1)^{a(q_0+p_1+p_2)} \cdot \frac{N_1^2 N_2^2}{M_1^2 M_2^2} \\
&=& (-1)^{a(q_0+p_1+p_2)} \cdot  \frac{b_1 b_1' a_2 a_2'}{a_1 a_1' b_2
  b_2'} \cdot \frac{a_1^2 b_2^2}{b_1^2 a_2^2}\\
&=&  (-1)^{a(q_0+p_1+p_2)} \cdot  \frac{a_1 b_2 b_1' a_2'}{b_1 a_2
  a_1' b_2'}.
\end{eqnarray*}
To get the final expression in Weber's formula, we now look for instance at the
linear system \eqref{sys1}. Using again Cramer's rule, one finds for
instance 
$$\frac{a_1}{b_1} =
\frac{[\beta_3,\beta_2,\beta_1]}{[\beta_{13},\beta_3,\beta_1]}, \; 
\frac{b_2}{a_2} =
\frac{[\beta_{13},\beta_3,\beta_{23}]}{[\beta_{3},\beta_2,\beta_{23}]}$$
and looking at \eqref{sys2}
$$\frac{b'_1}{a'_1} =
\frac{[\beta_{13},\beta_{12},\beta_1]}{[\beta_{12},\beta_2,\beta_1]}, \; 
\frac{a'_2}{b'_2} =
\frac{[\beta_{12},\beta_2,\beta_{23}]}{[\beta_{13},\beta_{12},\beta_{23}]}.$$
Changing the order of the columns, one gets the result.

\begin{remark} \label{rem:algo}
The complexity of the manipulations in this second part makes it
difficult to work out a  generalization of Weber's formula for
arbitrary genus. However, Remark
\ref{rem:generic} indicates that one should be able to  design  an
algorithm to compute the quotients of two Thetanullwerte in terms of
the 
equations of the hyperplanes supporting the odd theta characteristics
divisors. Indeed, if we denote $B_1,\ldots,B_{g-1}$ the support of
$D_{\bar{q_1}}$ and let $S'=A_2+\ldots+A_{g-1}+B_1+\ldots+B_{g-1}$,
then with the choice of sections of Remark \ref{rem:generic} we get that
$$  \frac{\vartheta[p_1](0)^4}{\vartheta[p_2](0)^4}
= (-1)^{a(q_0+p_1+p_2)} \cdot \frac{
  \chi_{2,S'}(A_1)^2}{\chi_{1,S'}(A_1)^2}.$$
This should be compared to a similar algorithm  suggested in
\cite{shepherdbarron-2008}. As far as we know, this latter version has
never been implemented.
\end{remark}

\subsection{A new proof} \label{sec:our}
In order to prove Weber's formula, we need  an extra
combinatoric result which proof can be easily obtained using the
results in Section \ref{sec:general}.
\begin{lemma} \label{aronq}
Let $q_S,q_T$ be two distinct even forms. Let $(q_1,\ldots,q_7)$ be an
Aronhold basis attached to $q_S$ ordered  such that $q_1+q_2+q_3=q_T$. Then
$$S'=(q_1',\ldots,q_7')=(q_{23},q_{13},q_{12},q_4,q_5,q_6,q_7)$$ is an
Aronhold basis attached to $q_T$ such that $q_1'+q_2'+q_3'=q_S$.
\end{lemma}

By the relation between Aronhold basis and fundamental
systems given in Proposition \ref{prop:connection} and applying Lemma \ref{aronq}, we get
\begin{lemma} \label{funda}
Let  $S=(q_1,\ldots,q_7)$ be an Aronhold basis attached to an even
characteristic $q_S$ and $q_1+q_2+q_3=q_T$. Then
$$P_0=(p_i)_{i=1,\ldots,8}=(q_1,q_2,q_3,q_{567},q_{467},q_{457},q_{456},q_S)$$
and 
$$P'_0=(p_i')_{i=1,\ldots,8}= (q_{23},q_{13},q_{12},q_{567},q_{467},q_{457},q_{456},q_T)$$
are fundamental systems.
\end{lemma}

Using Corollary \ref{cor:jacobi}  for the fundamental systems $P_0$ and $P'_0$
\begin{equation*} 
\frac{\left[p_1,p_2,p_3\right](\tau)}{\left[p_1^{'},p_2^{'},p_3^{'}\right](\tau)}
=
\frac{\left[q_1,q_2,q_3\right](\tau)}{\left[q_{23},q_{13},q_{12}\right](\tau)}
=\pm \prod_{i=4}^8
\frac{\vartheta\left[p_i\right](\tau)}{\vartheta\left[p_i^{'}\right](\tau)}
= \pm \frac{\vartheta[q_S](\tau)}{\vartheta[q_T](\tau)}.
\end{equation*}
Then Corollary \ref{link} shows that there exists constants
$\eta_i,\eta_{ij}$ (depending on $\beta_i,[q_i]$ or $\beta_{ij},[q_{ij}]$) 
such that
\begin{equation} \label{eq:S}
\frac{ [\beta_{1}, \beta_{2},\beta_{3}]}{[ \beta_{23},   \beta_{13}, \beta_{12}]} = \pm
\frac{ \eta_1 \eta_2 \eta_3 }{\eta_{23} \eta_{13} \eta_{12}} \cdot \frac{\vartheta[q_S](\tau)}{\vartheta[q_T](\tau)}.
\end{equation}
In order to kill the constants $\eta_i$, $\eta_{ij}$, we need to make
each $\beta_i,\beta_{ij}$ 
appears as many times in the numerator  as in the denominator. In
order to do this we
use Proposition \ref{prop:shift} to create new fundamental systems. To
simplify the notation and by analogy with the $q_{ij}$ let us denote $p_{ij} = p_{8}+ p_i +
p_j$ (for $1\leq i <j \leq 3$ we have $p_{ij}=q_{ij}$). For $1 \leq i
\leq 3$, let $v_i=p_8+p_i$, $v'_i= p'_8+p_i'$, $P_i=v_i+P_0$ and
$P_i'=v_i'+P_0'$. Since $v_i=v_i'$, we get the following explicit forms.
\begin{eqnarray*}
P_0 &=& (p_1,p_2,p_3,p_4,p_5,p_6,p_7,q_S), \\
P'_0&=& (p_{23},p_{13},p_{12},p_4,p_5,p_6,p_7,q_T), \\
P_1 &=& (p_1,p_{12},p_{13},p_{14},p_{15},p_{16},p_{17},q_S), \\
P_1' &=& (p_{23},p_{3},p_{2},p_{14},p_{15},p_{16},p_{17},q_T),\\
P_2 &=& (p_{12},p_{2},p_{23},p_{24},p_{25},p_{26},p_{27},q_S), \\
P'_2 &=& (p_{3},p_{13},p_{1},p_{24},p_{25},p_{26},p_{27},q_T), \\
P_3 &=& (p_{13},p_{23},p_{3},p_{34},p_{35},p_{36},p_{37},q_S), \\
P'_3 &=& (p_{2},p_{1},p_{12},p_{34},p_{35},p_{36},p_{37},q_T). 
\end{eqnarray*}
Hence 
\begin{eqnarray}
\frac{\left[\beta_1,\beta_{12},\beta_{13}\right](\tau)}{\left[\beta_{23},\beta_3,\beta_2\right](\tau)}
&=& \pm \frac{ \eta_1 \eta_{12} \eta_{13} }{\eta_{23} \eta_{3}
  \eta_{2}} \cdot
    \frac{\vartheta[q_S](\tau)}{\vartheta[q_T](\tau)},  \label{eq:S1} \\
\frac{\left[\beta_{12},\beta_2,\beta_{23}\right](\tau)}{\left[\beta_3,\beta_{13},\beta_1\right](\tau)}
&=& \pm \frac{ \eta_{12} \eta_2 \eta_{23} }{\eta_{3} \eta_{13}
  \eta_{1}} \cdot
    \frac{\vartheta[q_S](\tau)}{\vartheta[q_T](\tau)}, \label{eq:S2} \\ 
\frac{\left[\beta_{13},\beta_{23},\beta_3\right](\tau)}{\left[\beta_2,\beta_1,\beta_{12}\right](\tau)}
&=& \pm \frac{ \eta_{13} \eta_{23} \eta_3 }{\eta_{2} \eta_{1}
  \eta_{12}} \cdot \frac{\vartheta[q_S](\tau)}{\vartheta[q_T](\tau)}. \label{eq:S3}
\end{eqnarray}
Multiplying \eqref{eq:S},\eqref{eq:S1},\eqref{eq:S2} and \eqref{eq:S3} gives
 Weber's formula up to a sign which does not depend on $\tau$.

\subsection{The question of the sign} \label{sec:sign}
Following the different steps of the proof, we see that the sign in
Weber's formula only depends on the fundamental system $P_0$ and we
will denote it $\iota(P_0)$.
Let us denote also for a list of characteristics
$[P]=([p_1],\ldots,[p_8])$ such that
  $P=(p_1,\ldots,p_8)$ is a fundamental system and $\tau \in \HH_3$ 
 $$\calS([P],\tau) = \frac{[ p_1,  p_2,  p_3](
\tau)}{\prod_{i=4}^8 \theta[ p_i](
\tau)} = \pm 1.$$ 
When starting with a fundamental system $P$, we let $[P]$ be the
associated list of characteristics with coefficients $0$ and $1$.

\begin{lemma}[{\cite[p.420]{igusajacobi}}]
The following list $N_0=(n_1,\ldots,n_8)$ is a fundamental system (of
quadratic forms)
$$\begin{bmatrix} 1 & 0 & 0 \\ 1 & 0 & 0 \end{bmatrix}, \;
\begin{bmatrix} 0 & 1 & 0 \\ 1 & 1 & 0 \end{bmatrix}, \;
\begin{bmatrix} 0 & 0 & 1 \\ 1 & 1 & 1 \end{bmatrix}, \;
\begin{bmatrix} 1 & 0 & 0 \\ 0 & 0 & 0 \end{bmatrix}, \;
\begin{bmatrix} 0 & 1 & 0 \\ 1 & 0 & 0 \end{bmatrix}, \;
\begin{bmatrix} 0 & 0 & 1 \\ 1 & 1 & 0 \end{bmatrix}, \;
\begin{bmatrix} 0 & 0 & 0 \\ 1 & 1 & 1 \end{bmatrix}, \;
\begin{bmatrix} 0 & 0 & 0 \\ 0 & 0 & 0 \end{bmatrix}$$
\end{lemma}
For $1 \leq i \leq 3$, we can derive from $N_0$ the $N_i$, $N_0'$ and
the $N_i'$ as in Section \ref{sec:our}. For instance, we have for 
$N_0'=(n_1',\ldots,n_8')$ the following \emph{quadratic forms} 
$$\begin{bmatrix} 0 & 1 & 1 \\ 0 & 0 & 1 \end{bmatrix}, \;
 \begin{bmatrix} 1 & 0 & 1 \\ 0 & 1 & 1 \end{bmatrix}, \;
 \begin{bmatrix} 1 & 1 & 0 \\ 0 & 1 & 0 \end{bmatrix}, \;
 \begin{bmatrix} 1 & 0 & 0 \\ 0 & 0 & 0 \end{bmatrix}, \;  
\begin{bmatrix} 0 & 1 & 0 \\ 1 & 0 & 0 \end{bmatrix}, \;
 \begin{bmatrix} 0 & 0 & 1 \\ 1 & 1 & 0 \end{bmatrix}, \;%
 \begin{bmatrix} 0 & 0 & 0 \\ 1 & 1 & 1 \end{bmatrix}, \;
 \begin{bmatrix} 1 & 1 & 1 \\ 1 & 0 & 1 \end{bmatrix}.
$$

Using a compute algebra system like Magma\footnote{see \url{http://perso.univ-rennes1.fr/christophe.ritzenthaler/programme/theta-proof.magma}} \cite{magma}, we see that
\begin{lemma} \label{lem:n0}
$$\iota(N_0)=\prod_{i=0}^3 \frac{\calS([N_i],\tau)}{\calS([N_i'],\tau)} = 1.$$
\end{lemma}

\begin{remark}
One would rather  compute the sign using the classical trick
to evaluate the expression with $\tau$ a diagonal matrix. In this case one
can reduce the formula to expressions involving only genus $1$
Thetanullwerte and then use Jacobi identity.
 If this works well 
for $\calS([N_0],\tau)$, then for $\calS([N_0'],\tau)$ (for instance) the numerator and denominator are both
zero. Actually, because of the geometric meaning of the problem
--$\Jac(\calC)$ is an undecomposable principally polarized abelian variety--, it
seems that this will happen for any choice of $N_0$, as soon
as we consider a reducible $\tau$. This is why we had to adopt the
computational approach to get Lemma \ref{lem:n0}.
\end{remark}

We now want to understand what happens when we move to the given 
fundamental system $P_0$ we are interested in. For this purpose,
we will  need a transformation formula which we give
here for $g=3$. Up to identifying a characteristic
$[q]=\Ch{\eps}{\eps'}$ with the vector $\begin{pmatrix} ^t \eps \\ ^t \eps'\end{pmatrix}$, we let 
$\sigma = \begin{pmatrix} a & b \\ c & d \end{pmatrix} \in
\Sp_6(\ZZ)$ act by 
$$\sigma \cdot [q] = \begin{pmatrix} d & -c \\ -b &
  a \end{pmatrix} \begin{pmatrix}  ^t \eps \\  ^t \eps' \end{pmatrix}
+ \begin{pmatrix} (c ^t d)_0 \\ (a ^t b)_0 \end{pmatrix}.$$
Note that when we reduce modulo $2$, this action coincides with the
action of $\Gamma$ on quadratic forms as introduced in Section
\ref{rem:iden}. 
Let us also denote
$$\phi_{[q]}(\sigma)  = -\frac{1}{8} \left({ \eps} {^t b d} {^t \eps} - 2 {
  \eps} {^t bc } {^t \eps'} + { \eps'} {^t a c} {^t \eps'} -2 ^t (a ^t b)_0
({^t d}  {^t \eps} - c {^t \eps'})\right).$$
For a list of characteristics $[P]=([p_1],\ldots,[p_8])$ such that
$P=(p_1,\ldots,p_8)$ is a fundamental system, $\tau \in \HH_3$ and
$\sigma \in \Sp_6(\ZZ)$, let us denote $\sigma \cdot [P] = (\sigma
\cdot [p_1],\ldots,\sigma \cdot [p_8])$.
\begin{lemma}[{\cite[p.433]{igusajacobi}}] \label{lem:transfo}
With the notation above, we have
\begin{equation}\label{eq:transfo}
\calS(\sigma \cdot [P],\sigma \cdot \tau)  
= s([P],\sigma)
\cdot \calS([P],\tau)
\end{equation}
where $s([P],\sigma)= \kappa(\sigma)^{-2} \cdot \ex\left(\sum_{i=1}^3
\phi_{[p_i]}(\sigma) - \sum_{i=4}^8 \phi_{[p_i]}(\sigma)\right)$ and $\kappa(\sigma)$ is an $8$-th root of unity.
\end{lemma}

Let $P_0 = (p_0,\ldots,p_8)$ and let $\tilde{\sigma}
\in \Gamma$ be a matrix such that $\tilde{\sigma} \cdot n_i =
p_i$ for $1\leq i \leq 8$. Such a matrix always exists by the transitive action of
$\Gamma$ on fundamental systems. Let $\sigma \in \Sp_6(\ZZ)$ be any
lift of $\tilde{\sigma}$. If we apply $\sigma$ to the normalized
characteristics coming from the $N_i$ and $N_i'$, we get
characteristics for the forms in the $P_i$ and $P_i'$ because of the
linearity of the transformations involved in the definition of these
fundamental systems. Note that since  the lift of a given quadratic
form in the various fundamental systems $N_i,N_i'$ is fixed in the
various list of characteristics  $[N_i], [N_i']$ the
characteristics of the $[P_i]=\sigma \cdot [N_i]$, $[P_i']=\sigma
\cdot [N_i']$ have the same property. Moreover, even if
the
characteristics of the $[P_i]$ and $[P_i']$ are not necessarily normalized, we have already
noticed that the value of the global quotient does not change, since
all of them appear (twice) in the numerator and denominator. Because
of all these considerations, we get that 
\begin{eqnarray*}
\iota(P_0) &=&  \frac{\prod_{i=0}^3 \calS([P_i],\sigma \cdot \tau)}{\prod_{i=0}^3
  \calS([P'_i],\sigma \cdot \tau)} = \frac{\prod_{i=0}^3 \calS(\sigma \cdot [N_i],\sigma \cdot \tau)}{\prod_{i=0}^3
  \calS(\sigma \cdot [N'_i],\sigma \cdot \tau)} 
 = \frac{\prod_{i=0}^3   s([N_i],\sigma)}{\prod_{i=0}^3   s([N'_i],\sigma)}
\cdot \frac{\prod_{i=0}^3 \calS([N_i],\tau)}{\prod_{i=0}^3
  \calS([N'_i],\tau)} \\
&=& \frac{\prod_{i=0}^3   s([N_i],\sigma)}{\prod_{i=0}^3   s([N'_i],\sigma
)} 
= \frac{\ex(4 \cdot \phi_{[n'_8]}(\sigma))}{\ex(4 \cdot \phi_{[n_8]}(\sigma))} 
= (-1) ^{8 \cdot \phi_{[n'_8]}(\sigma)- 8 \cdot \phi_{[n_8]}(\sigma)}
\end{eqnarray*}
as all the characteristics apart from $[n_8]$ and $[n'_8]$ appear
twice in the numerator and the denominator. To finish the proof we
hence need the following lemma.
\begin{lemma}
$$8 \cdot \phi_{[n'_8]}(\sigma)- 8 \cdot \phi_{[n_8]}(\sigma) \equiv
 a(\sigma \cdot [n_8]  + \sigma \cdot [n'_8]+q_0) \pmod{2}.$$
\end{lemma}
\begin{proof}
Let $\sigma = \begin{pmatrix} a & b \\ c & d \end{pmatrix} \in
\Sp_6(\ZZ)$ and $[n'_8] = \Ch{\eps}{\eps'}$. The left hand side of
the expression is equivalent modulo $2$ to $r_1={ \eps} ^t b d {^t
  \eps} + { \eps'} {^t a c}{^t \eps'}.$
On the other hand 
$$[p_8] = \sigma \cdot [n_8] =\sigma \cdot
\Ch{000}{000}=  \begin{bmatrix}  (c ^t d)_0 \\  (a ^t
  b)_0 \end{bmatrix}$$
and 
$$[p'_8] = \sigma \cdot  [n'_8] = \begin{pmatrix} d & -c \\ -b &
  a \end{pmatrix} \begin{pmatrix} ^t \eps \\ ^t \eps' \end{pmatrix}
+ \begin{pmatrix} (c ^t d)_0 \\ (a ^t b)_0 \end{pmatrix}.$$
So  
\begin{eqnarray*}
[q] = [p_8]+[p'_8]+[q_0]&=&
\begin{pmatrix} d & -c \\ -b &
  a \end{pmatrix} \begin{pmatrix} ^t \eps \\ ^t \eps' \end{pmatrix}
  \equiv  \begin{pmatrix} d ^t \eps - c ^t \eps' \\ -b ^t \eps
  + a ^t \eps' \end{pmatrix} \pmod{2}.
\end{eqnarray*}
Finally 
\begin{eqnarray*}
a(q)& \equiv & ^t (d ^t \eps - c ^t \eps') { (-b ^t \eps
  + a ^t \eps')} \equiv \eps {^t b d} ^t \eps
+ \eps' {^t a c}
^t \eps' +  \eps ({^t b c} + {^t d  a}) ^t \eps' \\
&\equiv & r_1 + \eps ^t \eps'  \equiv r_1 + a(n'_8) \equiv r_1 \pmod{2}.
\end{eqnarray*}
\end{proof}

\bibliographystyle{abbrv}


\end{document}